\documentclass[11pt, letterpaper]{amsart}

\usepackage{esint}
\usepackage{amssymb}
\usepackage{MnSymbol}

\usepackage[%pagebackref,
hypertexnames=false, colorlinks, 
linkcolor=black,citecolor=black,urlcolor=black, linktocpage=true
]%
{hyperref} %,hypertexnames=false,colorlinks,[pagebackref]
\hypersetup{bookmarksdepth=3}

	\normalbaselines						  %

\usepackage{graphicx}
\usepackage{%stmaryrd,
mathrsfs}

\newtheorem{thm}{Theorem}[section]
\newtheorem{lm}[thm]{Lemma}

\theoremstyle{definition}
\newtheorem{df}[thm]{Definition}
\newtheorem*{df*}{Definition}

\theoremstyle{remark}
\newtheorem{rem}[thm]{Remark}
\newtheorem*{rem*}{Remark}

\numberwithin{equation}{section}

\newcommand{\ci}[1]{_{ {}_{\scriptstyle #1}}}
\newcommand{\nm }{\,\rule[-.6ex]{.13em}{2.3ex}\,}

\newcommand{\ti}[1]{_{\scriptstyle \text{\rm #1}}}
\newcommand{\ut}[1]{^{\scriptstyle \text{\rm #1}}}
\newcommand{\dd}{{\mathrm{d}}}

\newcommand{\trace}{\operatorname{trace}}

\newcommand{\wh}{\widehat}

\newcommand{\lla}{\llangle}
\newcommand{\rra}{\rrangle}

\newcommand{\cD}{\mathscr{D}}

\newcommand{\cM}{\mathcal{M}}

\newcommand{\cS}{\mathscr{S}}
\newcommand{\cG}{\mathcal{G}}
\newcommand{\cE}{\mathcal{E}}
\newcommand{\cF}{\mathcal{F}}

\newcommand{\fS}{\mathfrak{S}}
\newcommand{\f}{\varphi}

\newcommand{\e}{\varepsilon}

\newcommand{\C}{\mathbb{C}}
\newcommand{\R}{\mathbb{R}}
\newcommand{\Z}{\mathbb{Z}}
\newcommand{\N}{\mathbb{N}}
\newcommand{\E}{\mathbb{E}}
\newcommand{\F}{\mathbb{F}}
\newcommand{\shi}{\mathbb{S}}

\newcommand{\bI}{\mathbf{I}}

\newcommand{\ch}{\operatorname{ch}}
\newcommand{\ran}{\operatorname{Ran}}

\newcommand{\1}{\mathbf{1}}

\newcommand{\wt}{\widetilde}

\newcommand{\sign}{\operatorname{sign}}

\newcommand{\cz}{Calder\'{o}n--Zygmund\ }

\newcommand{\esssup}{\operatorname*{ess\,sup}}

\newcommand{\La}{\langle }
\newcommand{\Ra}{\rangle }

\newcommand{\bA}{\mathbf{A}}

\newcommand{\bL}{\mathbf{L}}

\newcommand{\la}{\langle}
\newcommand{\ra}{\rangle}

\newcommand{\supp}{\operatorname{supp}}

\def\cyr{\fontencoding{OT2}\fontfamily{wncyr}\selectfont}
\DeclareTextFontCommand{\textcyr}{\cyr}
\newenvironment{entry}
{\begin{list}{X}%
  {%
      \setlength{\labelwidth}{55pt}%
      \setlength{\leftmargin}{\labelwidth}%\labelsep}%
      \addtolength{\leftmargin}{\labelsep}%
   }%
}%
{\end{list}}      

\renewcommand{\labelenumi}{(\roman{enumi})}

\newcounter{vremennyj}

\newcommand\cond[1]{\setcounter{vremennyj}{\theenumi}\setcounter{enumi}{#1}\labelenumi\setcounter{enumi}{\thevremennyj}}

\begin{document}
\title[Convex body domination]%
{
Convex body domination and weighted estimates with matrix weights
%Calder\'on--Zygmund 
%operators
}
%\author[A. ~Lerner]{Andrei Lerner}
%\thanks{}
%\address{Department of Mathematics, Bar Ilan University, Ramat Gan, Israel}

\author[F.~Nazarov]{Fedor Nazarov}
\thanks{}
\address{Department of Mathematics, Kent State University, USA}
\email{nazarov@math.kent.edu \textrm{(F. Nazarov)}}
\author[S.~Petermichl]{Stefanie Petermichl}
%\thanks{AV is partially supported by the NSF grant DMS-1265549 and by the Hausdorff Institute for %Mathematics, Bonn, Germany}
\address{Department of Mathematics, Universit\'e Paul Sabatier, Toulouse, France}
\email{stefanie.petermichl@gmail.com \textrm{(S. Petermichl)}}
\author[S.~Treil]{Sergei Treil}
\thanks{}
\address{Department of Mathematics, Brown University, USA}
\email{treil@math.brown.edu \textrm{(S. Treil)}}
\thanks{Work of S.~Treil is supported by the NSF grants DMS-1301579, DMS-1600139}
\author[A.~Volberg]{Alexander Volberg}
\thanks{AV  %and VV are 
is partially supported  by the Oberwolfach Institute for Mathematics, Germany; AV is also supported by the NSF grant DMS-1265549}
\address{Department of Mathematics, Michigan Sate University, East Lansing, MI. 48823}
\email{volberg@math.msu.edu \textrm{(A. Volberg)}}
\makeatletter
\@namedef{subjclassname@2010}{
  \textup{2010} Mathematics Subject Classification}
\makeatother
\subjclass[2010]{42B20, 42B35, 47A30}
% 42B	Harmonic analysis in several variables
% 42B20	Singular and oscillatory integrals (Calder?on-Zygmund, etc.)
% 42B35	Function spaces arising in harmonic analysis
% 47A	General theory of linear operators
% 47A30	Norms (inequalities, more than one norm, etc.)
%{30E20, 47B37, 47B40, 30D55.}
%
% 30D55	$H^p$-classes (1980-2009)
% 30E20	Integration, integrals of Cauchy type, integral representations of analytic functions
%
% 47B   	Special classes of linear operators
% 47B37	Operators on special spaces (weighted shifts, operators on sequence spaces, etc.)
% 47B40	Spectral operators, decomposable operators, well-bounded operators, etc.
\keywords{matrix weights
   martingale transform, matrix weighted maximal function}
\begin{abstract}
We introduce the so called \emph{convex body valued} sparse operators, which generalize the notion of sparse operators to the case of spaces of vector valued functions.

We prove that \cz operators as well as Haar shifts and paraproducts can be dominated by such operators.  By 
estimating sparse operators we obtain weighted estimates with matrix weights. We get two weight $A_2$-$A_\infty$ estimates, that in the one weight case give us the estimate 
\[
\|T\|\ci{L^2(W)\to L^2 (W)} \le C [W]\ci{\bA_2}^{1/2} [W]\ci{A_\infty} \le C[W]\ci{\bA_2}^{3/2}
\] 
where $T$ is either \cz operator (with modulus of continuity satisfying the Dini condition), or a Haar shift or a paraproduct. 
\end{abstract}
\maketitle

\setcounter{tocdepth}{1}
\tableofcontents
\setcounter{tocdepth}{3}

\section*{Notation}
%\addcontentsline{toc}{section}{Notation}
\begin{entry}

\item[$|Q|$] for $Q\subset \R^N$ denotes its $N$-dimensional Lebesgue measure;\\
\item[$\cD$] a dyadic lattice. We consider all ``translations'' of the standard dyadic lattice;\\
\item[$\la f\ra\ci Q$] average of the function $f$ over $Q$, $\la f \ra \ci Q:= |Q|^{-1} \int_Q f(x) \dd x$;\\ 
\item[$\lla f \rra\ci Q$] ``convex body valued'' average of a functions $f$ with values in $\R^d$, see Section \ref{s:CB_av}; \\
\item[$\|\,\cdot\,\|, \nm\,\cdot\, \nm$]  norm; since we are dealing with matrix- and operator-valued functions we will use the symbol $\|\,\cdot\,\|$ (usually with a
subscript) for the norm in a functions space, while $\nm\,\cdot\,\nm$ is used for the
norm in the underlying vector (operator) space.
 Thus for a vector-valued function
$f$ the symbol
$\|f\|_2$ denotes its
$L^2$-norm, but the symbol $\nm f\nm$ stands for the scalar-valued function  $x\mapsto \nm f(x)\nm$;  \\
\end{entry}

\section{Motivations, definitions and results}
\label{nr}

This paper started as an (unsuccessful) attempt to prove the so-called $A_2$-conjecture for the weighted estimates with matrix weights.  

Recall that a ($d$-dimensional) matrix weight on $\R^N$ is a locally integrable function on $\R^N$ with values in the set of positive-semidefinite $d\times d$ matrices. The weighted space $L^2(W)$ is defined as the space of all measurable functions $f:\R^N \to \F^d$, (here $\F=\R$, or $\F=\C$) for which
\[
\| f\|\ci{L^2(W)}^2:=\int (W(x)f(x), f(x)) \dd x <\infty\,;
\]
here $(\cdot, \cdot)$ means the usual duality in $\F^d$.

A matrix weight $W$ is said to satisfy the matrix $\mathbf A_2$ condition (write $W\in(\mathbf A_2)$) if 
\[
[W]\ci{\bf A_2} := \sup_Q \nm \la W\ra\ci Q^{1/2} \la W^{-1}\ra\ci Q^{1/2} \nm^2  < \infty\,.
\]
The quantity $[W]\ci{\bf A_2}$ is called the \emph{$\mathbf A_2$ characteristic of the weight} $W$.

In \cite{TV} it has been proved that the weighted estimate
\[
\| Tf\|\ci{L^2(W)} \le C \|f\|\ci{L^2(W)} 
\]
holds  for  the Hilbert transform $T$  (or for the Haar multipliers) if and only if $W\in(\bA_2)$ (for necessity we need to assume that for no vector $e\in\F^d$ we have $W(x)e=0$ a.e.)

Moreover, it has been proved in \cite{BPW} that for the Hilbert transform
\begin{equation}
\label{log}
\| T\|\ci{L^2(W)\to L^2(W)}\le c(d) [W]_{\bA_2}^{3/2} \log (1+[W]_{\bA_2} )\,.
\end{equation}
However in the scalar case $d=1$ there is just  $c [w]_{A_2} $ in the right hand side of \eqref{log}; such estimate in the scalar case is now proved for a wide class of \cz operators, as well as for their martingale analogues, and the constant $c$ there depends only on the operator $T$, but not the weight. 

In the scalar case this was the instance of the famous $A_2$ conjecture proved first in \cite{W},  \cite{PV}, \cite{P}, and then in full generality by \cite{Hy}.
Let us mention that after \cite{Hy} many reproofs appeared one more elegant than the other, see, e. g. \cite{HyPTV}, \cite{Le}.

A natural question then would be whether it is possible to have only $C(T, N, d) [W]\ci{\bA_2}$ on the right hand side of \eqref{log}, or whether in the matrix case there are some new phenomena and the linear norm estimate in terms of $[W]\ci{\bA_2}$ fails. We still are not able to answer this question, we only manage to eliminate the logarithm $\log[W]\ci{\bA_2}$, leaving us with the exponent $3/2$. 

However, this is not the main results of the paper. 

One of  the main results of this paper is a theorem about domination of vector-valued \cz operators (and of their dyadic analogues) by sparse operators. In the scalar case, domination by sparse operators significantly simplified the proof of the $A_2$ and $A_p$ conjectures, and allowed to extend it to the most general class of \cz operators, namely to the case of $\omega$-\cz operators with the modulus of continuity $\omega$ satisfying the Dini condition. 

In this paper we introduce a notion of domination by a sparse operator for operators in vector-valued spaces,   that can be considered a ``correct'' generalization of the scalar case. Our sparse operator $\bL\ci\cS$ acts to the space of function whose values are symmetric  convex sets in $\R^d$, and the ``domination'' means the inclusion 
\[
Tf(x) \in \bL\ci\cS f(x)\qquad \text{a.e.~on } \R^N. 
\] 
And essentially, our first result is that if a scalar operator can be dominated by a sparse operator, then its vector version (i.e., its tensor product with the identity $\bI_d$ in $\F^d$) can be dominated by our \emph{convex body valued} sparse operator. 

We were not able to prove the  result in such generality,  but we have proved it  for all scalar operators that are known to admit domination by sparse operators, i.e., for $\omega$-\cz operators with the modulus of continuity satisfying the Dini condition,  and for a wide class of martingale operators, including the so-called \emph{big Haar shifts} and \emph{paraproducts}, see the definitions in Section \ref{s:HaarShift} below. 

The convex body valued sparse operators look complicated, but the weighted estimates of these operators can be done via very simple \emph{scalar} operators. In this direction we were able to obtain some $A_2$-$A_\infty$ type weighted estimates with matrix weights, even in a two-weight setting. 

Namely, assuming that $W(x)$ is invertible a.e. and denoting for example by $M\ci W^{1/2}$ multiplication by $W^{1/2}$,  an operator $T$ acts in $L^2(W)$ if and only if $M\ci W^{1/2} T M\ci{W}^{-1/2}$ is bounded in the non-weighted $L^2$. Thus it is a natural problem to consider the two-weight problem of finding the condition on matrix weights $V$, $W$ such that the operator $M\ci W^{1/2} T M\ci{V}^{1/2}$ is bounded (in the non-weighted $L^2$). 

We assume that the weights $V$ and $W$ satisfy the two-weight matrix $\bA_2$ condition 
\begin{align}
\label{2w-A_2}
\sup_Q  \nm\la W\ra\ci Q^{1/2} \la V\ra\ci Q^{1/2}\nm^2 =: [W, V]\ci{\bA_2} <\infty;
\end{align}
here the supremum is taken over all cubes in $\R^N$. 
This assumption seems natural, because 
acting the same way as in \cite{TV} it is possible to show for the Hilbert transform $T$ this condition \eqref{2w-A_2} is necessary.  

Recall that a scalar weight $w$ on $\R^N$ is said to satisfy the $A_\infty$ condition if for all cubes $Q\subset \R^N$
\begin{align}
\label{A_infty}
\la M\ci Q  w  \ra\ci Q \le C \la w\ra\ci Q,
\end{align}
where $M$ is the maximal function adapted to the cube $Q$
\begin{align}
\label{M_Q}
M\ci Q f(x) = \sup \{ |\la f \ra\ci R | \,: \, R\in \cD(Q), x\in R\}
\end{align}
(we put $M\ci Q f(x) =  0$ for $x\notin Q$). 

The best constant in \eqref{A_infty} is called the $A_\infty$ characteristic of the weight $w$, and denoted by $[w]\ci{A_\infty}$. 

For a matrix weight $W$ define its \emph{scalar $A_\infty$ characteristic} $[W]\ci{A_\infty}\ut{sc}$ as 
\begin{align}
\label{A_infty-sc}
[W]\ci{A_\infty}\ut{sc}:= \sup_{e\in\F^d} [w_e ]\ci{A_\infty}, 
\end{align}
where the scalar weight $w_e$ is defined by $w_e(x) = (W(x) e, e)$, $x\in \R^N$. 

It is well known and will be explained later in the paper that $[W]\ci{A_\infty}\ut{sc} \le C(N) [W]\ci{\bA_2}$. 

\begin{thm}
\label{t:estCZ}
Let $T$ be a \cz operator with modulus of continuity $\omega$ satisfying the Dini condition. Assume that the weights $V$, $W$ satisfy the joint $\bA_2$ condition and that they both satisfy the scalar $A_\infty$ condition, meaning that $[V]\ci{A_\infty}\ut{sc}, [W]\ci{A_\infty}\ut{sc} <\infty$. 
Then 
\[
\| M\ci W^{1/2} T M\ci V^{1/2} \|\ci{L^2\to L^2}^2 \le C [V]\ci{A_\infty}\ut{sc} [W]\ci{A_\infty}\ut{sc} [V, W]\ci{\bA_2}, 
\]
where $C=C(T, N, d)$. 
\end{thm}

For a dyadic lattice $\cD$ in $\R^N$ one can define corresponding dyadic $\bA_2^\cD$ and $A_\infty^\cD$ classes, by considering in \eqref{2w-A_2} and \eqref{A_infty} only cubes $Q\in\cD$. For a matrix weight we can also define the scalar dyadic $\bA_\infty^\cD$ class. 

We will use $[W, V]\ci{\bA_2^\cD}$, $[w]\ci{A_\infty^\cD}$, $[W]\ci{\bA_\infty^\cD}\ut{sc}$ for the corresponding characteristics.

As another application of convex body domination we immediately get the following weighted estimate for norms of dyadic operators, namely \emph{big  Haar shifts} and \emph{paraproducts} mentioned above, see the definitions in Section \ref{s:HaarShift}.

\begin{thm}
\label{t:estHS}
Let $T$ be a big Haar  shift or a paraproduct (with respect to a dyadic lattice $\cD$).  Assume that the weights $V$, $W$ satisfy the joint $\bA_2^\cD$ condition and that they both satisfy the scalar $A_\infty ^ \cD$ condition, meaning that $[V]\ci{A_\infty^\cD}\ut{sc}, [W]\ci{A_\infty^\cD}\ut{sc} <\infty$. 
Then 
\[
\| M\ci W^{1/2} T M\ci V^{1/2} \|\ci{L^2\to L^2}^2 \le C [V]\ci{A_\infty^\cD}\ut{sc} [W]\ci{A_\infty^\cD}\ut{sc} [V, W]\ci{\bA_2^\cD}, 
\]
where $C=C(T, N, d)$. 
\end{thm}

\begin{rem}
As we discussed above, the boundedness of an operator $T$ in the weighted space $L^2(W)$ is equivalent to the boundedness of the operator $M^{1/2}_W T M^{-1/2}_W$ in the non-weighted space $L^2$. 

It is well known (and is explained here in Section \ref{ss:MatrW}, see Remark \ref{r:A_infty-A_2} there) that $[W]\ci{A_\infty}\ut{sc} \le C [W]\ci{\bA_2}$, so for the one weight estimates Theorems \ref{t:estCZ} and \ref{t:estHS} give $C [W]\ci{\bA_2}^{3/2}$ for the estimates of the norm. 
\end{rem}

\section{Convex body domination of singular integral operators}
In the rest of the paper we will treat $\C^n$ as a real vector space, so all vector functions will be $\R^d$-valued. 
\subsection{What is a sparse family?}
\label{s:sparse_fam}

There are several definition of sparse family of cubes. 

\begin{df}[Classical definition]
\label{df:Sparse} Let $0<\e<1$
A collection of dyadic cubes $\cS\subset \cD$ is called $\e$-\emph{sparse} if for any $Q\in\cS$ 
\[
\sum_{R\in\ch_\cS Q} |R| \le \e |Q|. 
\]
\end{df}

\begin{df}
\label{df:w-Sparse} Let $0<\eta < 1$. 
A collection $\cS$ of cubes (not necessarily dyadic) is called (weakly) $\eta$-sparse if there exists a disjoint collection of measurable sets $E\ci Q\subset Q$, $Q\in\cS$ such that
\[
| E\ci Q |\ge \eta |Q|
\]
for all $Q\in\cS$. 
\end{df}

\begin{df}[Dyadic Carleson family]
\label{df:dyCarlFam}
Let $\lambda >0$. A collection $\cS$ of dyadic cubes is called \emph{dyadic $\lambda$-Carleson} if for any $Q\in\cS$
\[
\sum_{R\in\cS(Q)} |R| \le \lambda |Q|. 
\]
\end{df}

\begin{df}[Carleson family]
\label{df:CarlFam}
Let $\Lambda >0$. A collection $\cS$ of  cubes (not necessarily dyadic) is called $\Lambda$-\emph{Carleson} if for any cube $Q$
\[
\sum_{\substack{R\in\cS(Q)\\ \ell(R)\le\ell(Q)%, Q\cap Q\ne\varnothing 
} } |R\cap Q| \le \Lambda |Q|.
\]
\end{df}

\subsubsection{Comparison of different definitions of sparse families}

First note, that if $\cS$ is a family of dyadic cubes, then it is Carleson in the sense of both Definition \ref{df:dyCarlFam} and Definition \ref{df:CarlFam}. Moreover, the best constants $\Lambda$ from Definition \ref{df:CarlFam} and $\lambda$ from  Definition \ref{df:dyCarlFam} are equivalent, 
\[
\lambda \le \Lambda \le 2^N \lambda. 
\]

For an $\e$-sparse (in the sense of Definition \ref{df:Sparse}) dyadic system $\cS$ one can define 
for $Q\in\cS$
\[
E\ci Q := Q\setminus \bigcup_{R\in\ch_\cS Q} R, 
\]
so such system is trivially (weakly) $\eta$-sparse in the sense of Definition \ref{df:w-Sparse}. 

It is also easy to see that a dyadic weakly $\eta$-sparse family (in the sense of Definition \ref{df:w-Sparse}) is a dyadic $\lambda$-Carleson family (in the sense of Definition \ref{df:dyCarlFam}) with $\lambda=1/\eta$. 

The converse is also true: any dyadic $\lambda$-Carleson family is $\eta$-sparse (in the sense of Definition \ref{df:w-Sparse}) with $\eta =1/\lambda$, see \cite[Lemma 6.3]{LN}.

It is also obvious that any dyadic $\lambda$-Carleson family with $\lambda<2$ is $\e$-sparse (in the sense of Definition \ref{df:Sparse}) with $\e=\lambda-1$.

It is also clear that given a dyadic $\lambda$-Carleson family ($\lambda$ is assumed to be large) one can split it into $n=n(\lambda, \lambda_1)$ $\lambda_1$-Carleson families, where $\lambda_1>1$ can be chosen as close to $1$ as we want. 

The proof is quite easy if one does not care about constants. If one cares about constants, it was proved in \cite[Lemma 6.6]{LN} that for any natural $m\ge 2$ a dyadic $\lambda$-Carleson family could be split into $m$ $\lambda_1$-Carleson families with $\lambda_1 = 1+ (\lambda-1)/m$. 

Finally, the standard ``three lattice trick'' allows as to estimate a sparse operator with respect to a $\Lambda$-Carleson family by a sum of $3^N$ dyadic $\lambda$-Carleson sparse operators, corresponding to $3^N$ dyadic lattices; here $\lambda=\lambda(\Lambda, N)$ and the estimate is with the constant $C=C(N)$. The estimate is trivial for both scalar sparse operators and for the convex body valued sparse operators defined below, see \eqref{CBLern}.  For the latter operators the domination   means inclusion. 

So, if we are interested in weighted  estimates of operators, it really does not matter what type of sparse families we are using. For example, if we dominate an operator by a $\Lambda$-Carleson sparse operators, the weighted estimates for this operator would follow from the estimates of classical $\e$-sparse dyadic operators (for $3^N$ dyadic lattices) with some $\e>0$. And again, this works for both scalar and convex body valued sparse operators. 

\subsection{What is an average of a vector-valued function and a sparse operator?}
\label{s:CB_av}
For a function $f\in L^1(Q)$ with values in $\R^d$ define its (convex body) average $\llangle f\rrangle\ci Q$ as 
\begin{align}
\label{C-B-Av}
\lla f \rra\ci Q:= \{ \La \f f \Ra\ci Q \Bigm| \f: Q\to \R, \ \|\f\|_\infty\le 1 \}. 
\end{align}
Clearly, $\lla f\rra\ci Q$ is a symmetric, convex compact set (it is closed because the closed unit ball in $L^\infty$ is weak* compact).

For a sparse family $\cS$ of cubes define a sparse (Lerner) operator $\bL=\bL\ci{\cS}$ by 
\begin{align}
\label{CBLern}
\bL\ci\cS f = \sum_{Q\in\cS} \lla f\rra\ci Q \1\ci Q,  
\end{align}
where the sum is understood as Minkowsky sum. 

We do not specify here what we mean by a sparse family, since any of the above definitions of a sparse family can be used. 

\begin{lm}
For a sparse family $\cS$ of cubes and for compactly supported $f\in L^1$, a.e.~on $\R^N$ the set $\bL\ci\cS f(x)$ is a bounded convex symmetric subset of $\R^d$. 
\end{lm}

\begin{proof}
For $f\in L^1\ti{loc}(\R^N;\R^d)$ and for any cube $Q\subset \R^N$, the  set $\lla f \rra\ci Q$ is a bounded, convex, symmetric subset of $\R^d$. Thus the fact that for all $x\in \R^N$ the set $\bL\ci\cS f(x)$ is convex and symmetric follows immediately. 

But if $\cS$ is a sparse family, for almost all $x\in\R^N$ only finitely many $Q\in\cS$ such that the side length of $Q$ is bounded by one may contain $x$, so the a.e.~boundedness of the part of the sum $\bL\ci\cS f(x)$, where the summation goes over  $Q\in\cS$ such that the side length of $Q$ is bounded by one follows immediately. We are left to consider the part of the sum of $\bL\ci\cS f(x)$, where the summation is over only ``big" cubes  (such that the side length of $Q$ is at least two).

Without loss of generality we can assume that the compact support of $f$ lies only in one dyadic cube of side length two. Otherwise we split $f$ to finitely many functions having this property by using that its support is compact. We call this cube $Q\ci f$.

  Let $\cF$ be the collection of all dyadic cubes of side length  two or larger that intersect the support of $f$.  As $f\in L^1$ we can see immediately that 
\[
\sum\ci{Q\in \cF} \nm\langle \varphi_Q f\rangle\ci Q\nm\le  \sum\ci{Q\in \cD, Q_f\subset Q} \frac1{|Q|}\int\ci{Q\ci f} \nm f\nm \dd x\le 4^d\sum\ci{k=0}^\infty 2^{-dk} \|f\|_1\,.
\]
So the part of the sum of $\bL\ci\cS f(x)$, where the summation is over only ``big" cubes is uniformly bounded for such $f$.
\end{proof}

\subsubsection{John ellipsoids}
An ellipsoid in $\R^d$ is an image of the closed unit ball $B$ in $\R^d$ under a non-singular affine transformation. 

Recall, that for a convex body (i.e.~a compact convex set with non-empty interior) $K$  in $\R^d$ its \emph{John ellipsoid} is an ellipsoid of maximal volume contained in $K$. It is known that the John ellipsoid is unique, and that if $K$ is also symmetric, then its John ellipsoid $\cE=\cE\ci K$ is centered at $0$ and 
\begin{equation}
\label{JohnE}
\cE\subset K \subset \sqrt d \cE. 
\end{equation}

In the construction we will need  John ellipsoids for the sets $\lla f\rra\ci Q$. However, $\lla f\rra\ci Q$ does not have to have non-empty interior. 

\begin{lm}
\label{l:John-01}
Let $f\in L^1(Q)$ be non-trivial (i.e. $f(x)\ne 0$ on a set of positive measure). Then there exists unique subspace $E\subset \R^d$  containing $\lla f\rra\ci Q$ such that  $\lla f\rra\ci Q$ has non-empty interior in  $E$. 
\end{lm}

\begin{proof}
The function $f\in L^1(Q)$ gives rise to a continuous non-zero linear map $F:L^{\infty} \to \R^d$ by $F(\varphi)=\la \varphi f \ra\ci Q$. Taking the subspace $E=\ran F$ of $\R^d$ and applying the open mapping theorem we conclude that the image of the open unit ball in $L^{\infty}$ is an open set in $E$.
\end{proof}

So, for a set $\lla f\rra\ci Q$ its John ellipsoid is defined as John ellipsoid in the subspace $E$.

\subsection{How to estimate convex set-valued sparse operators}
\label{s:CB-int}
Our sparse operators $\bL\ci\cS$ look like  very complicated objects, but the estimates of such operators is rather simple. 

Everything is based on the following simple lemma:

\begin{lm}
\label{l:ConvAv}
Let $f\in L^1(Q, \R^d)$ and let $g(x)\in \lla f\rra\ci Q$ a.e.~on $Q$. Then there exists a measurable function $K:Q\times Q\to \R$, $\|K\|_\infty \le |Q|^{-1}$ such that 
\[
g(x) = \int_{Q} K(x,y) f(y) \dd y. 
\]
\end{lm}

\begin{proof}
The statement is trivial if $g$ is a simple function (i.e.~a measurable function taking finitely many values). For a general $g$, approximating it by simple functions $g_n$, $g_n \rightrightarrows g$ and taking a weak* limit point (say in $L^\infty$) of the corresponding kernels $K_n$ completes the proof. 
%
%Namely, if $g_n= \int_Q K_n (x,y) f(y) \dd y $, and $K_{n_k}$ converges to $K$ in the weak* topology of $L^\infty$, we get that for any $h\in L^1$
%%%
%\begin{align*}
%\iint_{Q\times Q} K(x,y) h(x) f(y) \dd x \dd y &= \lim_{k\to \infty} \iint_{Q\times Q} K_{n_k}(x,y) h(x) f(y) \dd x \dd y
%\\
%& = \lim_{k\to\infty} \int_Q g_{n_k}(y) h(y) \dd y = \int_Q g(y) h(y) \dd y 
%\end{align*}
%%%
\end{proof}

Using this lemma we can see that to estimate a convex body-valued sparse operator $\bL=\bL\ci\cS$
\[
\bL\ci\cS f = \sum_{Q\in\cS} \lla f\rra\ci Q\1\ci Q, 
\]
one needs to find a uniform bound on all operators of form 
\begin{align}
\label{sparse-03}
f\mapsto \sum_{Q\in\cS} \int K\ci Q (x,y) f(y) dy, 
\end{align}
where kernels $K\ci Q$ are supported on $Q\times Q$ and satisfy $\| K\ci Q\|_\infty \le |Q|^{-1}$. 

The latter problem lies in the realm of harmonic analysis. 

\medskip

Notice that the statement of Lemma \ref{l:ConvAv} can be pushed a little bit further. Namely, we can claim the following.

\begin{lm}
\label{l:ConvAvrank}
Let $f\in L^1(Q, \R^d)$. Then there exist real measurable functions $\{\varphi_i\}_{i=1}^d$ supported on $Q$, $\|\varphi_i\|_\infty\le 1$ such that for any $g$ such that  $g(x)\in \lla f\rra\ci Q$ a.e.~on $Q$ there exist   real measurable functions $\{\psi_i\}_{i=1}^d$, $\|\psi_i\|_\infty \le C(d)$, such that 
\[
g(x) =\sum_{i=1}^d \psi_i(x) \langle \varphi_i f\rangle\ci Q. 
\]
\end{lm}
\begin{proof}
We assume first that John ellipsoid $\cE$ of $\lla f\rra\ci Q$ is $d$-dimensional. Let $\{g_i\}_{i=1}^d$, $g_i\in\cE$ be vectors corresponding to its principal axis. Since $g_i\in \cE\subset \lla f\rra\ci Q$,
there exist real functions $\varphi_i$ supported on $Q$ that
\[
g_i= \langle \varphi_i f\rangle_Q, \,\qquad  \|\varphi\|_\infty \le 1, \, i=1, \dots, d\,.
\]
On the other hand every measurable vector function $g$ on $Q$ with values in $\lla f\rra\ci Q$ has the form
\[
g(x)= \sum_{i=1}^d a_i^{(g)}(x) g_i, \, \,\text{a.e}\, \,x\,,
\]
where $\{a_i^{(g)}(x)\}$ are measurable functions, and $\sum_{i=1}^d |a_i^{(g)}(x)|^2 \le d$. This is by \eqref{JohnE}.
If $\cE$ has dimension less than $d$ we just need less than $d$ vector functions $g_i$, so we can choose the rest of $g_i'$s to be zero.
Therefore, the claim of the lemma follows.
\end{proof}

\begin{rem} Using this lemma we can see that to estimate a convex body-valued sparse operator $\bL=\bL\ci\cS$
\[
\bL\ci\cS f = \sum_{Q\in\cS} \lla f\rra\ci Q\1\ci Q, 
\]
one needs to find a uniform bound on all operators of form 
\[
f\mapsto \sum_{Q\in\cS} \psi\ci Q(x)\langle \phi\ci Q f \rangle\ci Q, 
\]
where real functions  $\varphi\ci Q, \psi\ci Q$ are supported on $Q$ and satisfy $\| \varphi\ci Q\|_\infty \le 1, \| \varphi\ci Q\|_\infty \le 1$.  

In other words, in estimating $\bL\ci\cS f $ we can always think about estimating uniformly operators 
\eqref{sparse-03}
%$f\mapsto \sum_{Q\in\cS}  \int K\ci Q (x,y) f(y) dy,$ 
with the extra property that all $K\ci Q$ are rank one operators.
\end{rem}

\begin{rem} In terms of bilinear estimates we can rephrase the previous remark as follows. To estimate $| (\bL\ci\cS f , g)|$ it is sufficient to have an estimate of  bilinear forms 
\[
\sum_{Q\in\cS} |\langle\psi\ci Q g\rangle\ci Q\langle \varphi\ci Q f \rangle_Q| |Q|
\]
uniform in real functions $\varphi\ci Q, \psi\ci Q$  having their $L^\infty$ norm bounded by one.
\end{rem}

\section{Domination of vector-valued singular integral operators by sparse operators}
The main result of this section is in Subsection \ref{s:Sc-vec}. Subsections \ref{s:dom-CZO} and \ref{s:HaarShift} essentially just give a different presentation of known results. These subsections are presented simply for the reader's convenience. 
\subsection{From scalar to vector domination}
\label{s:Sc-vec}
Informally speaking, if a scalar operator $T$ can be dominated by a sparse one, the same should hold for its  vector-valued version $T\otimes \bI_d$. Unfortunately, we are not able to prove a general theorem to that extend. However we are able to prove that  a scalar induction step (that can be used to prove the sparse domination in all known scalar cases) implies the corresponding induction step for vector valued operators.

We will need the following definition
\begin{df}
Let $\cG$ and $\overline \cG$ be two collections of disjoint dyadic cubes. We say that $\overline\cG$ covers (is covering)  $\cG$ if  for any $Q\in \cG$ one can find $R\in\overline\cG$ such that $Q\subset R$. 
\end{df}
In the  the language of stopping times, this just means a pointwise earlier stopping time. 

The lemma below is universal for any sensible linear operator $T$, such as the Calder\'{o}n--Zygmund operator, Haar shift or paraproduct. 
%More general filtrations can be considered in the case of the 0 shift. We restrict ourselves to dyadic filtrations here, having in mind the Calder\'{o}n--Zygmund operator as main application. 

\begin{lm}
\label{lemma_scalar_to_vector}
For $r\ge 1$ and  $Q\in\cD$ denote  $Q'=rQ$. 

Let $T$ be a linear operator such that for any $\e>0$ and for  any  $f\in L^1(\R^N; \R)$, supported on the cube $Q'_0$  there exists a collection $\cG%\={Q_k\}_{k=1}^\infty
$  of disjoint dyadic subcubes of $Q_0$ satisfying
\begin{enumerate}
\item $\displaystyle \sum_{Q\in\cG} |Q| \le \e |Q_0|$. 

\item $\displaystyle \Bigl|Tf(x) - \sum_{Q\in\cG}  \1\ci{Q} T( f \1\ci{Q'})\Bigr| \le C \e^{-1}\La |f| \Ra\ci{Q'_0}     $ a.e.~on $Q_0$, where $C=C(T, N)$ does not depend on $f$. 

\item For any collection $\overline\cG$ of disjoint dyadic subcubes of $Q_0$ that covers $\cG$
\[
\Bigl|Tf(x) - \sum_{Q\in\overline\cG}  \1\ci{Q} T( f \1\ci{Q'})\Bigr| \le C \e^{-1}\La |f| \Ra\ci{Q'_0}  \qquad\text{a.e.~on }Q_0. 
\]
\end{enumerate}

Then for any $0<\delta<1$ and for any vector-valued functions $f\in L^1(\R^N;\R^d)$, supported on $Q'_0$  there exists a family $\cG_1$ of disjoint dyadic subcubes of $Q_0$ such that 
\begin{align}
\label{MeasStop-02}
\sum_{Q\in\cG_1} |Q|& \le \delta |Q_0| 
\intertext{and}
\label{SpDom-step-02}
T f (x) & \in C \lla f\rra\ci{Q'_0} + \sum_{Q \in\cG_1} \1\ci{Q}T( f \1\ci{Q'})\qquad \text{a.e.~on } Q_0, 
\end{align}
where  $C=C(T, N, d, \delta)$  (here we slightly abuse notation and use $T$ instead of $T\otimes \bI_d$). 
\end{lm}
The assumptions of this lemma are essentially the properties used in the induction step in the construction of sparse domination in \cite{L_2016}, only they are written in a slightly different way. 

The case $r=1$ will be used to get domination for dyadically localized operators, like Haar shifts and paraproducts. The case $r>1$ can be used to get the domination operators that are not dyadically localized, like  \cz operators.

\begin{proof}[Proof of Lemma \ref{lemma_scalar_to_vector}]
Consider the representation of the John ellipsoid $\cE$ of $\lla f \rra\ci{Q'_0}$ in principal axes, i.e.~let $e_1, e_2, \ldots e_d$ be an orthonormal basis in $\R^d$ and $\alpha_k\in [0,\infty)$ such that 
\begin{align}
\label{John-01}
\cE=\Big\{\sum_{k=1}^d x_k \alpha_k e_k : x_k\in\R, \ \sum_{k=1}^d x_k^2 \le 1 \Bigr\} .
\end{align}

Let $f_k(x):= (f(x), e_k)\ci{\R^d}$. Since $\lla f\rra\ci{Q'_0} \subset \sqrt{d}\cE$, one can conclude that $\La |f_k|\Ra\ci{Q'_0} \le \sqrt{d}\alpha_k$ (consider $\La \f_k f_k\Ra\ci{Q'_0}$ with $\f_k=\sign f_k$).

Applying the hypothesis with $\e = \delta d^{-1}$ to each $f_k$, we will get for each $f_k$ a collection $\cG_k$ of dyadic subcubes of $Q_0$ such that a.e.~on $Q_0$
\begin{align*}
|Tf_k(x)| \le C \sqrt{d} \alpha_k d + \sum_{Q\in\cG_k}  \1\ci{Q} |T( f_k \1\ci{Q'})| ;
\end{align*}
here we used the estimate $\La |f_k|\Ra\ci{Q'_0} \le \sqrt{d}\alpha_k$. 

Let $\cG$ be the collection of maximal cubes in the collection $\bigcup_{k=1}^d \cG_k$. Since $\cG$ covers any of $\cG_k$, part \cond3 of the hypothesis implies that for all $k$ we have a.e.~on $Q_0$
\begin{align*}
|Tf_k(x)| \le C d\sqrt{d} \alpha_k + \sum_{Q\in\cG}  \1\ci{Q} |T( f_k \1\ci{Q'})| . 
\end{align*}

Then clearly a.e.~on $Q_0$
\begin{align*}
Tf(x) \in C d\sqrt{d} P + \sum_{Q\in\cG}  \1\ci{Q} T( f \1\ci{Q'})
\end{align*}
where $P$ is the ``box''
\begin{align*}
P:= \Bigl\{\sum_k x_k\alpha_k e_k : x_k\in[-1, 1]\Bigr\}. 
\end{align*}
Since trivially $P\subset \sqrt{d}\cE$, where $\cE$ is the John ellipsoid \eqref{John-01}, we get that a.e.~on $Q_0$
\begin{align*}
Tf(x) \in Cd^2 \cE + \sum_{Q\in\cG}  \1\ci{Q} T( f \1\ci{Q'}) \subset 
Cd^2 \lla f\rra\ci{3Q_0}  + \sum_{Q\in\cG}  \1\ci{Q} T( f \1\ci{Q'}) .
\end{align*}
Noticing that 
\begin{align*}
\sum_{Q\in\cG} |Q| \le \sum_{k=1}^d \sum_{Q\in\cG_k} |Q| \le d \cdot (2d)^{-1} =1/2
\end{align*}
completes the proof. \end{proof}

\begin{rem}
\label{r:ArbFiltr}
One can see from the proof that Lemma \ref{lemma_scalar_to_vector} holds not just for a dyadic filtration, but for any \emph{atomic} filtration, i.e.~a filtration where on each step a ``cube'' $Q$ splits into finitely (or countably) many ``cubes''. In particular, this lemma holds for any collection $\cD^{k,r}:=\bigcup_{j\in\Z} \cD_{k+(r+1)j}$
\end{rem}

\subsection{Domination of \cz operators}
\label{s:dom-CZO}

Let us recall some definition. Let $\omega$ be a \emph{modulus  of continuity}, i.e. an increasing subadditive function on $[0,\infty)$ satisfying $\omega(0)=0$.  A bounded (in a scalar $L^2=L^2(\R^N)$) operator $T$ is called an $\omega$-\cz operator, if for any $f\in L^2$ and $x\notin\supp f$
\[
Tf(x) = \int_{\R^N} K(x,y) f(y) dy
\]
and the kernel $K$ satisfy the following size and smoothens conditions
\begin{align*}
|K(x,y) | &\le C |x-y|^{-N}  \\
|K(x, y) - K(x' y)| + |K(y, x) - K(y, x') | &\le \omega \left( \frac{|x-x'|}{|x-y|} \right) |x-y|^{-N} \quad \text{if }|x-x'|\le 2 |x-y|  .
\end{align*}

We say that the modulus of continuity $\omega$ satisfies the Dini condition if
\[
\|\omega\|\ti{Dini} := \int_0^1 \omega(x) \frac{\dd x }{x} <\infty. 
\]

\begin{thm}
\label{t:SpDom-01}
Let $T$ be an $\omega$-\cz operator with modulus of continuity $\omega$ satisfying the Dini 
condition.

Then for a compactly supported $f\in L^1(\R^N;\R^d)$ there exists an $\eta$-sparse (in the sense of Definition \ref{df:w-Sparse}) family $\cS=\cS(f)$ with $\eta = 3^{-N} /2$ and such that 
%
%There exist $3^N$ dyadic lattices $\cD_n$, $n=1,2, \ldots 3^N$, such that  given a compactly supported function $f\in L^1(\R^d)$  there exists $M$ $\frac12$-sparse dyadic families $\cS_k$ (each $\cS_k\subset \cD_n$ for some $n$)%
%\footnote{Usually $M\ge 3^N$.}
%%
% such that for almost all $x\in\R^d$
%%
\begin{align}
\label{SpDom-01}
Tf(x) \in C  \bL\ci{\cS} (x),  
\end{align}
where %sparse families $\cS_k$ depend on the function $f$, but $M$ and 
the constant $C$ depends only on the operator $T$ and dimensions $N$ and $d$. 
\end{thm}

%The sparse family $\cS$ is constructed by induction, where the induction step is given by the following lemma
%
%\begin{lm}
%\label{l:SpDom-step-01}
%Let $f\in L^1(\R^N;\R^d)$ be supported on a cube $3Q_0$. There exist a family $\cG_1$ of disjoint dyadic subcubes of $Q_0$ such that 
%%%
%\begin{align}
%\label{MeasStop-02}
%\sum_{Q\in\cG_1} |Q|& \le \frac12 |Q_0| 
%\intertext{and}
%\label{SpDom-step-02}
%T f (x) & \in C \lla f\rra\ci{3Q_0} + \sum_{Q \in\cG_1} \1\ci{Q}T( f \1\ci{3Q})\qquad \text{a.e.~on } Q_0, 
%\end{align}
%%%
%where the constant $C$ depends only on the operator $T$ and dimensions $N$ and $d$. 
%\end{lm}

\subsubsection{Proof of Theorem \ref{t:SpDom-01}}

The hypothesis of Lemma \ref{lemma_scalar_to_vector} for Calder\'{o}n--Zygmund operators with $r=3$ (i.e.,~$Q'=3Q$) was essentially proved in \cite{L_2016}, see estimate (3.4) there. It was stated for $\e=1/2$, but the proof works for arbitrary $\e$. Hypothesis \cond3 of the lemma was not explicitly proved in \cite{L_2016}, but can be easily seen from the proof there. For the convenience of the reader we present a proof of the hypotheses of Lemma \ref{lemma_scalar_to_vector} for the case of Calder\'{o}n-Zygmund operators, essentially Lerner's argument, in Section \ref{s:LernArg} below.

Assume that the hypotheses of Lemma \ref{lemma_scalar_to_vector} are satisfied for any dyadic cube. 

 Take a compactly supported $f\in L^1(\R^N;\R^d)$ and a cube $Q_0$, $\supp f\subset \frac12 Q_0$. Applying Lemma \ref{lemma_scalar_to_vector} with $r=2$ and $\delta=1/2$ we get the family $\cG_1$ of dyadic subcubes of $G_0$ such that \eqref{MeasStop-02} and \eqref{SpDom-step-02} hold. 

We then apply Lemma \ref{lemma_scalar_to_vector} to each cube $Q\in\cG_1$ (with function $\1\ci{3Q} f$) to get the family $\cG_2$, and so on. 

Trivially, the family $\cG:=\{Q_0\}\bigcup_{n\ge 1} \cG_n$ is a dyadic $\e$-sparse family with $\e=1/2$, and so it is $\eta$-sparse family in the sense of Definition \ref{df:w-Sparse} with $\eta=1/2$. Since
\[
\lim_{n\to\infty} \sum_{Q\in\cG_n} |Q|  = 0, 
\]
we can conclude that a.e.~on $Q_0$
\begin{align*}
 T f (x) \in C \sum_{Q\in \cG} \lla f \rra\ci{3Q}\1\ci Q(x) 
\end{align*}
To dominate $Tf(x)$ outside of $Q_0$, we notice that for $n\ge 0$ and $x\in 3^{n+1}Q_0 \setminus 3^{n}Q_0$
\[
Tf(x) \in C \lla f\rra\ci{3^{n+1}Q_0}, 
\]
so a.e.~on $\R^N$
\begin{align}
\label{SpDom-02a}
T f (x) \in C \sum_{Q\in \cG} \lla f \rra\ci{3Q}\1\ci Q(x) + C\sum_{n\ge 1} \lla f\rra\ci{3^n Q_0} \1\ci{3^n Q_0}. 
\end{align}
Note, that the inclusion will hold if in the first sum we replace $\1\ci{Q}$ by $\1\ci{3Q}$ (the right hand side will be bigger). As we discussed before, the collection $\cG$ is a dyadic $\eta$-sparse family with $\eta=1/2$, so the collection $\{3Q:Q\in\cG\}$ is $\eta$-sparse family with $\eta= 3^{-N} /2$. If we add to this collection cubes $3^n Q_0$, $n\ge 2$, it will remain $\eta$-sparse (with the same $\eta=3^{-N}/2$). 

So the collection $\cS:= \{3Q:Q\in\cG\}\cup\{3^nQ_0:n\ge 2\}$ is $\eta$-sparse, and \eqref{SpDom-01} trivially holds because of \eqref{SpDom-02a}.

Thus we proved Theorem \ref{SpDom-01}, assuming that hypotheses of Lemma \ref{lemma_scalar_to_vector} are satisfied. \hfill\qed

\subsubsection{Assumptions  of Lemma \ref{lemma_scalar_to_vector} are satisfied for \cz operators } 
\label{s:LernArg}
The proof below is borrowed from \cite{L_2016}. We present it here only for the reader's convenience. 

Consider the maximal operator $\cM\ci T$, introduced in \cite{L_2016}, 
\begin{align}
\label{M_T}
\cM\ci T f (x) := \sup_{Q:x\in Q} \esssup_{\xi\in Q} |T (\1\ci{\R^N \setminus 3Q} f) (\xi)|
\end{align}
We need the following Lemma, see \cite[Lemma 3.2]{L_2016}

\begin{lm}
\label{l:Lern-3.2}
Let $T$ be an  $\omega$-\cz operator  with $\omega$ satisfying the Dini condition. Then for $f\in L^1(\R^N)$ we have a.e.~on $Q_0$
\begin{enumerate}
\item $\displaystyle |T(\1\ci{3Q_0} f(x)| \le   C_N \|T\|\ci{L^1\to L^{1,\infty}} |f(x)| + \cM\ci T f(x)$.   
\item $\displaystyle\cM\ci T f(x) \le C\ci N (\|\omega\|\ti{Dini} + C(T) ) Mf(x) + T^\sharp f(x) 
%= C_1 Mf(x) + T^\sharp f(x)
$. 
\end{enumerate}
Here $M$ is the Hardy--Littlewood maximal operator and $T^\sharp$ is the maximal truncation of $T$, 
\[
T^\sharp f(x) = \sup_{\e>0} \Bigl|\int_{|x-y|>\e} K(x,y) f(y) dy \Bigr| .
\]
\end{lm}
If $\omega $ satisfies the Dini condition, then the operators $T$ and $T^\sharp$ are of weak type $1$-$1$; the maximal operator $M$ also is of weak type $1$-$1$. Therefore the operator $\cM\ci T$  is also of weak type $1$-$1$.

So, there exist constants $C_{1,2}=C_{1,2}(T, N)$ such that for any  $\e>0$ the measure of the set $E_\e\subset Q_0$, 
\begin{align*}
E_\e:=\bigl\{ x\in Q_0 : \cM\ci T f(x) > C_1 \e^{-1} \La |f|\Ra\ci{3Q_0} \bigr\}\cup 
\bigl\{ x\in Q_0 : | f(x) | > C_2 \e^{-1} \La |f|\Ra\ci{3Q_0} \bigr\}
\end{align*}
satisfies the estimate $|E|\le 2^{-N-1} \e |Q_0|$. 

Statement \cond1 of Lemma \ref{l:Lern-3.2} implies that on $Q_0\setminus E_\e$
\begin{align}
\label{est-T-01}
|Tf(x) | \le C_3 \e^{-1} \La |f|\Ra\ci{3Q_0}
\end{align}
for an appropriate constant $C_3=C_3(T, N)$. 

Now define $\cG$ as the collection of maximal dyadic subcubes $Q$ of $Q_0$ such that 
\begin{align*}
\La \1\ci{E_\e}\Ra\ci{Q} > 2^{-N-1} %\La |f|\Ra\ci{3Q_0}
\end{align*}
Since $E_\e\subset \bigcup_{Q\in\cG} Q$, the estimate \eqref{est-T-01} holds a.e.~on $Q_0\setminus \bigcup_{Q\in\cG} Q$.

Note that for $Q\in\cG$ we have $\La \1\ci{E_\e}\Ra\ci{Q} \le 2^{-1}$, because $\La \1\ci{E_\e}\Ra\ci{R} \le 2^{-1-N}$ on the parent $R$ of $Q$.  But that means $\cM\ci T f(x_0) \le C_1 \e^{-1}\La |f|\Ra\ci{3Q_0}$ for some $x_0\in Q$. 

Therefore $|T (f\1\ci{\R^N\setminus 3Q})(x) | \le C_1 \e^{-1}\La |f|\Ra\ci{3Q_0}$ a.e.~on $Q$, because otherwise the inequality $\cM\ci T f(x)> C_1 \e^{-1}\La |f|\Ra\ci{3Q_0}$ would hold everywhere on $Q$. 
Thus, statement \cond2 holds with $C=\max\{C_1, C_3\}$. 

To prove statement \cond3 we just notice that if $\overline\cG$ covers $\cG$ then still $E_\e\subset \bigcup_{Q\in\overline \cG} Q$, and that $\La \1\ci{E_\e}\Ra\ci{Q} \le 2^{-1}$ for any $Q\in \overline\cG$. So the same proof as for $\cG$ works for $\overline\cG$.  
\hfill \qed

\subsection{Domination of vector-valued  Haar shifts and paraproducts by sparse operators}
\label{s:HaarShift}

Recall, that a \emph{generalized big Haar shift} of complexity $r\ge0$ is a bounded in $L^2(\R^N)$ operator $\shi$
\begin{align}
\label{H-shift-01}
\shi = \sum_{Q\in\cD} T\ci Q, \qquad T\ci Q f(x) =\int K\ci Q(x,y) f(y) dy,  
\end{align}
where kernels $K\ci Q$ are supported on $Q\times Q$,  constant on all $R\times S$ with $R, S\in\ch^{r+1} Q$ and satisfy the estimate 
\begin{align}
\label{K_HS-01}
\| K \|_\infty \le |Q|^{-1} .
\end{align}
We say that $\shi$ is a big Haar shift, without the word \emph{generalized}, if, in addition $T\ci Q\1\ci Q =0$, $T^*\ci Q\1\ci Q=0$.

For a locally integrable function $b$ paraproduct $\Pi_b^r$ of order $r$ with symbol $b$ is defined by
\begin{align}
\label{Para-01}
\Pi_b^r f =\sum_{Q\in\cD} \La f \Ra\ci Q \sum_{R\in\ch^r Q} \Delta\ci R b. 
\end{align}
Note, that if $\|\Pi_b^r\|\le 2^{-Nr/2}$, then $\Pi_b^r$ is a generalized big Haar shift. 

\begin{df}
\label{df:r-sep}
A generalized big Haar shift is said to be $r$-separated if there exists $k=0, 1, 2, \ldots , r$ such that  $T\ci Q\ne 0$ only if $Q\in \cD^k=\cD^{k,r}:=\bigcup_{j\in\Z}\cD_{k+(r+1)j}$. 
\end{df}

Each generalized big Haar shift of complexity $r$ can be represented as a sum of $r+1$ $r$-separated ones, so it is sufficient to estimate only $r$-separated Haar shifts. 

Note also that if $T$ is an $r$-separated  generalized big Haar shift, then with respect to the lattice $\cD^k$ it will be a shift of complexity $1$.

\begin{thm}
\label{t:DomHS}
Let $T$ be either a big Haar shift of complexity $r$ or a paraproduct $\Pi_b^r$ of order $r$ (with $\|\Pi_b^r\|\le 2^{-Nr/2}$, so it is a generalized big Haar shift), and let $T$ be $r$-separated. Given $\delta \in (0,1)$ there exists $C=C(\e, N, d)$ such that for every compactly supported $f\in L^1(\R^N;\R^d)$ there exists an $\e$-sparse family $\cS\subset \cD$  (in the sense of Definition \ref{df:Sparse}) such that
\begin{align}
\label{SpDom-02}
Tf(x) \in C  \bL\ci{\cS} (x),  
\end{align}
where $C=C(N,d, \e)$.
\end{thm}

The theorem can be easily obtained from the lemma below. For the $r$-separated shift $T$ from Theorem \ref{t:DomHS} we denote   $\cD^k$  from Definition \ref{df:r-sep} by $\wt \cD:= \cD^k$, skipping the index $k$.

\begin{lm}
\label{l:SpDom-03-step-03}
Let $T$ be as in Theorem \ref{t:DomHS}. Given $\e\in (0,1)$ there exists $C=C( N)$ such that for any function $f\in L^1(\R^N)$ supported on $Q_0\in \wt\cD$ there exists a collection $\cG$ of disjoint cubes $Q\in\wt\cD$, $Q\subset Q_0$ such that 
\begin{enumerate}
\item $\displaystyle  \sum_{Q\in\cG} |Q| \le \e |Q_0|     $;
\item $\displaystyle \Bigl|Tf(x) - \sum_{Q\in\cG}  \1\ci{Q} T( f \1\ci{Q})\Bigr| \le C \e^{-1}\La |f| \Ra\ci{Q_0} $ a.e.~on $Q_0$;
\item For any disjoint collection $\overline\cG\subset \wt\cD(Q_0)$ that covers $\cG$
\[
\Bigl|Tf(x) - \sum_{Q\in\overline\cG}  \1\ci{Q} T( f \1\ci{Q})\Bigr| \le C \e^{-1}\La |f| \Ra\ci{Q_0}  \text{a.e.~on } Q_0 .
\]
\end{enumerate}
\end{lm}

\subsubsection{Proof of Theorem \ref{t:DomHS} }
The above Lemma \ref{l:SpDom-03-step-03} says that the hypotheses of Lemma \ref{lemma_scalar_to_vector} are satisfied for all cubes $Q\in\wt\cD$. Applying Lemma \ref{l:SpDom-03-step-03} with $r=1$ and with  $\wt \cD$ instead of $\cD$, see Remark \ref{r:ArbFiltr}, we get the following lemma.

\begin{lm}
\label{l:SpDom-03-step-01}
Let $T$ be as in Theorem \ref{t:DomHS}. Given $\e\in (0,1)$ there exists $C=C(\e, N)$ such that for any function $f\in L^1(\R^N; \R^d)$ supported on $Q_0\in \wt\cD$ there exists a collection $\cG_1$ of disjoint cubes $Q\in\wt\cD$, $Q\subset Q_0$ such that 
\begin{enumerate}
\item $\displaystyle  \sum_{Q\in\cG_1} |Q| \le \e |Q_0|     $;
\item $\displaystyle  T f (x)  \in C \lla f\rra\ci{Q_0} + \sum_{Q \in\cG_1} \1\ci{Q}T( f \1\ci{Q})$ a.e.~on $Q_0$. 
\end{enumerate}
\end{lm}

To prove Theorem \ref{t:DomHS}  we iterate Lemma \ref{l:SpDom-03-step-01} to get that for a function $f\in L^1(\R^N;\R^d)$ supported on $Q_0\in\wt\cD$
\begin{align*}
T f(x) \in C\sum_{Q\in\cS^0} \lla f\rra\ci Q \1\ci Q, \qquad \text{a.e.~on } Q_0, 
\end{align*}
for a sparse family $\cS^0\subset \wt\cD(Q_0)$. To estimate $Tf$ outside of $Q_0$, we need to estimate $\sum_{R\in\wt\cD: Q_0\subsetneqq R} T\ci R$. But  $T\ci R f(x) \in \1\ci R(x)  \lla f \rra\ci R$ a.e., and for each $R\in\wt\cD$, $Q_0\subsetneqq R$ we have 
\[
\lla f \rra\ci R = \frac{|Q_0|}{|R|} \lla f \rra\ci{Q_0}. 
\]
So adding to $\cS^0$ the cubes $R_k\in\wt\cD$, $k\ge 1$, where $Q_0\subsetneqq R_k$ and for all $k\ge 1$
\[
|R_{k+1} |/|R_k| = |R_1|/|Q_0|  \le\e, 
\]
we get the conclusion of the theorem for the operator $T$. \hfill\qed

\subsubsection{Proof of Lemma \ref{l:SpDom-03-step-03} }
First recall that an $r$-separated generalized big Haar shift $T$ of complexity $r$ with $\|T\|\le 1$  has weak type $1$-$1$, and that
\begin{align}
\label{weak-02}
\|T\|\ci{L^1\to L^{1, \infty}} \le C = C(N), 
\end{align}
see \cite[Theorem 5.2]{HyPTV}. 

Define $\cG$ to be the collection of maximal cubes $R\in\wt\cD$ such that either of two conditions below holds
\begin{align}
\label{stop-03}
\Bigl| \sum_{Q\in\wt\cD \,:\, R\subsetneqq Q} T\ci Q(f) \Bigr| & > 2{C}\e^{-1} \La |f|\Ra\ci{Q_0} \qquad \text{on } R, \qquad 
\\
\label{stop-04} 
\La |f|\Ra\ci{R} & > 2 \e^{-1}\La |f|\Ra\ci{Q_0} , 
\end{align}
where $C=C(N)$ is from \eqref{weak-02}.

We claim that $\sum_{Q\in\cG} |Q| \le \e |Q_0|$. Let $\cG_1\subset G$ be the collection of stopping cubes where \eqref{stop-03} holds, and let $\cG_2=\cG\setminus\cG_1$.

Consider the operator $T^1$, 
\[
T^1 = \sum_{Q\in\wt\cD\setminus\bigcup_{R\in\cG} \wt\cD(R)} T\ci Q .
\]
By \eqref{stop-03} on any cube $R\in\cG_1$ we have $|T^1 f(x)| > C\e^{-1}$, 
so the weak type estimates for $T^1$ and disjointness of $R\in\cG$ imply 
\[
\sum_{R\in\cG_1} |R|  = \Bigl| \bigcup_{R\in\cG_1} R \Bigr| \le (\e/2)|Q_0| . 
\]

 \begin{rem*}
Here we used the estimate \eqref{weak-02} of the weak $1$-$1$ norm, that depends only on $\|T\|_2$ and $N$. We can use it, since for the operators from Theorem \ref{t:DomHS} the truncation does not increase the norm. 
 \end{rem*}   

Since for any $R\in \cG_2$ we have $\La |f|\Ra\ci{R}  > 2 \e^{-1}\La |f|\Ra\ci{Q_0}$, the trivial weak type estimates imply that 
\[
\sum_{R\in\cG_2} |R|  = \Bigl| \bigcup_{R\in\cG_2} R \Bigr| \le (\e/2)|Q_0|, 
\]
and statement \cond1 is proved.

Let us now prove statement \cond2. 
It follows from the construction, see stopping condition \eqref{stop-03} that 
\begin{align}
\label{ii-01}
|T f(x)| %, \ |T^1f(x)| 
\le C\e^{-1}  \La |f| \Ra\ci{Q_0}, \qquad \text{for a.e. }  x\in Q_0\setminus \bigcup_{R\in\cG} R .
\end{align}
Let $R\in \cG$ and let $\hat R$ be its $\wt\cD$-parent. 
Again, it follows from the construction, that on $R$
\begin{align}
\label{ii-02}
%\Bigl| T^{Q_0} f(x) -T^R f(x) \Bigr| = 
\Bigl| \sum_{Q\in\wt\cD \,:\, \hat R\subsetneqq Q} T\ci Q(f) \Bigr| \le  C \e^{-1}\La |f| \Ra\ci{Q_0} .
\end{align}
For the shift $T=\sum_Q T\ci Q$ and a cube $R\in\wt\cD$ define 
\[
T^{R}:= \sum_{Q\in\wt\cD(R)} T\ci Q .
\]     
For $R\in \wt\cD$ we write
\[
\1\ci R Tf -\1\ci R T(\1\ci R f) =  \left( \1\ci R Tf - T^R f\right)  + \left( T^R f - \1\ci R T(\1\ci R f) \right)
\]
and estimate each term separately. 

To estimate the first term notice that for $x\in R\in \cG$
\[
T f(x) - T^R f(x) = \sum_{Q\in\wt\cD \,:\, \hat R\subsetneqq Q} T\ci Q f (x) \quad  + \quad T\ci{\hat R} f(x). 
\]
The sum is estimated in \eqref{ii-02}. To estimate $T\ci{\hat R} f(x)$, recall that by the construction $\La |f|\Ra\ci{\hat R} \le 2 \e^{-1}  \La |f|\Ra\ci{Q_0}$, so recalling that $\|K\ci{\hat R}\|_\infty \le |\hat R |^{-1}$, we get that 
\[
| T\ci{\hat R} f(x) | \le \La |f|\Ra\ci{\hat R} \le 2 \e^{-1}  \La |f|\Ra\ci{Q_0},  
\]
so we get the desired estimate of $\1\ci R Tf - T^ R f$. 

To estimate $T^R f - \1\ci R T(\1\ci R f) $ we write
\[
 \1\ci R T(\1^R f) - T^R f  =  \sum_{Q\in\wt\cD \,:\, \hat R\subsetneqq Q} \1\ci R T\ci Q (1\ci R f) .
\]
Since for $x\in R$, and $Q\supsetneqq \hat R$  we have 
\[
|T\ci Q (\1\ci Q f)(x)| \le |Q|^{-1} |\hat R | \cdot \La |\1\ci R f |  \Ra\ci{\hat R}
\le  |Q|^{-1} |\hat R | \cdot \La | f |\Ra\ci{\hat R}. 
\]
It follows from the stopping condition \eqref{stop-04} that $\La | f |\Ra\ci{\hat R} \le 2 \e^{-1}\La |f|\Ra\ci{Q_0}$, so summing the geometric progression we get that $| T^R f - \1\ci R T(\1\ci R f) | \le C \e^{-1} \La |f|\Ra\ci{Q_0} \1\ci R$. Thus 
\[
| \1\ci R Tf -\1\ci R T(\1\ci R f) | \le C \e^{-1} \La |f|\Ra\ci{Q_0} \1\ci R, 
\]
which together with \eqref{ii-01}  gives statement \cond2  of the lemma. 

To prove \cond3, we notice that \eqref{ii-01} and \eqref{ii-02} hold and the above construction works if we replace $\cG$ by any collection $\overline\cG\subset \wt\cD$ of disjoint cubes that covers $\cG$. \hfill\qed

%\section{Weighted spaces with matrix weights}
%
%Recall that a matrix weight $V$ is a locally integrable function with values in the set of positive semi-definite $d\times d$ Hermitian matrices (consider only real matrices????). The weighted space $L^p(V)$ is defined as the set of all functions $f:\R^N\to \R^d$ such that 
%%%
%\begin{align}
%\label{L^p(V)-01}
%\|f\|\ci{L^p(V)}^p := \int_{\R^N} \nm V(x)^{1/p}f(x)\nm^p \dd x <\infty
%\end{align}
%%%
%(of course one needs to take the quotient space over the subspace of all functions $f$, $\|f\|\ci{L^p(V)}=0$). 
%
%
%For $p=2$ formula \eqref{L^p(V)-01} can be simplified, 
%%%
%\begin{align}
%\label{L^2(V)-01}
%\|f\|\ci{L^2(V)}^2 := \int_{\R^N} (V(x) f(x), f(x) )\ci{\R^d} \dd x, 
%\end{align}
%%%
%but not for other $p$. 
%
%Sometimes 
%
%\vdots
%
%\vdots
%
%
%\vdots
%

\section{Some known facts about \texorpdfstring{$A_2$}{A<sub>2} and \texorpdfstring{$A_\infty$}{A<sub>infty} weights.}
\label{prelim}

We will need  two well-known facts on scalar weights and one fact on matrix weights.  Scalar weights will be denoted by $w$, matrix weights by $W$.

\subsection{Comparison of \texorpdfstring{$A_2$}{A<sub>2} and \texorpdfstring{$A_\infty$}{A<sub>infty} weights and reverse H\"{o}lder inequality for \texorpdfstring{$A_\infty$}{A<sub>infty} weights.}

The first fact is very simple:

\begin{lm}
\label{a2}
If $w\in A_2^\cD$, then for any $Q\in\cD$ we have
\begin{equation}
\label{triv}
\int_Q M\ci Q  w \, \dd x \le 4 [ w]\ci{A_2^\cD} \int_Q w \, \dd x\,.
\end{equation}
\end{lm}
\begin{proof} For any $R\in\cD(Q)$
\[
[ w]\ci{A_2^\cD} \ge \La w \Ra\ci R \La w^{-1} \Ra\ci R = \left( \La w \Ra\ci R \exp( - \La \ln w \Ra\ci R )\right)
\left( \La w^{-1} \Ra\ci R \exp(  \La \ln w \Ra\ci R ) \right). 
\]
By Jensen inequality both factors  in the right hand side are at least $1$, 
therefore 
\[
\La w \Ra\ci R \exp( - \La \ln w \Ra\ci R ) \le [ w]\ci{A_2^\cD},  
\]
and thus
\[
\La w \Ra\ci R \le [ w]\ci{A_2^\cD} \exp(\La \ln w \Ra\ci R ) \le [ w]\ci{A_2^\cD}  \La  w^{1/2} \Ra\ci R^2; 
\]
in the last inequality we have used the Jensen inequality again. 

Then $M\ci Q w \le [ w]\ci{A_2^\cD} (M\ci Q w^{1/2})^2$, and using the $L^2$ estimate for the maximal function we get 
\[
\int_Q M\ci Q  w \, \dd x \le [ w]\ci{A_2^\cD} \int_Q (M\ci Q  w^{1/2} )^2\, \dd x
\le 4 [ w]\ci{A_2^\cD} \int_Q w \, \dd x\,.
\]
\end{proof}

%\medskip
%
%Recall that the best constant in the inequality
%\begin{equation}
%\label{trivC}
%\forall Q,\,\int_Q M(1_Qw) dx \le C \int_Qw dx\,.
%\end{equation}
%is called (weak) $A_\infty$ characteristic of $w$, and it is denoted by $[w]_{A^w_\infty}$. 
%
%\bigskip
%
%\noindent{\bf Definition.}
%The weights with  $[w]_{A^w_\infty}<\infty$ are called $A_\infty$ weights. 
%
%\bigskip

The above Lemma \ref{a2} immediately implies that 
\begin{equation}
\label{trivA}
[w]\ci{A^\cD_\infty} \le 4 [w]\ci{A_2^\cD}\, , \qquad  [w]\ci{A_\infty} \le 4 [w]\ci{A_2}.
\end{equation}

The next fact is more subtle, it is proved in \cite{Va} by the Bellman function method and in \cite{HyPeR} by a stopping time argument.

\begin{thm}
\label{RH}
Let $w\in A_\infty^\cD$ and let $0<\delta\le  \frac{2^{-N-1}}{ [w]_{A^\cD_\infty}}$. Then for any $Q\in\cD$ 
\begin{equation}
\label{RHineq}
\la w^{1+\delta}\ra\ci Q \le 2 \la w\ra\ci Q^{1+\delta}\,.
\end{equation}
\end{thm} 

This theorem was proved in \cite[Theorem 2.3]{HyPeR}. It was assumed there that $w\in A_\infty$, but only the fact that $w\in A_\infty^\cD$ was used in the proof.

\subsection{Some properties of matrix weights}
\label{ss:MatrW}

\begin{lm}
\label{l:vec-sc-A_2}
Let $W\in \bA_2$. For $e\in \R^d$ define $w_e(x) := (W(x) e, e)\ci{\R^d}$. Then $w_e\in A_2$ and 
\[
[w_e]\ci{A_2} \le [W]\ci{\bA_2}. 
\]
The same lemma holds if we replace $A_2$ and $\bA_2$ by $A_2^\cD$ and $\bA_2^\cD$. 
\end{lm}
\begin{proof}
The fact is well-known, cf.~\cite{TV}. The easiest proof is probably to recall that for the averaging operator $\E\ci Q$, $\E\ci Q f = \la f \ra\ci Q \1\ci Q$ its norm in $L^2(W)$ can be computed as 
\[
\| \E\ci Q\|\ci{L^2(W) \to L^2(W)} = \nm \la W\ra\ci Q^{1/2} \la W^{-1}\ra\ci Q^{1/2}\nm . 
\]
Restricting $\E\ci Q$ to functions of form $\f e$, where $\f$ is a scalar valued function we prove the lemma. 
\end{proof}

\begin{rem}
\label{r:A_infty-A_2}
Combining lemmas \ref{l:vec-sc-A_2} and \ref{a2} we can see that for the matrix weights $[W]\ci{A_\infty}\ut{sc} \le 4 [W]\ci{\bA_2}$, and the same holds for the dyadic versions.  
\end{rem}

\begin{lm}
Let $w_k\in A_\infty$, $k=1,2, \ldots, n$, and let $[w_k]\ci{A_\infty} \le A$. 
Then for $w:= \sum_{k=1}^n w_k $ we have $[w]\ci{A_\infty} \le A$. 

The same lemma holds with $A_\infty$ replaced by $A_\infty^\cD$. 
\end{lm}
\begin{proof}
The estimates $[w_k]\ci{A_\infty} \le A$ means that for each $Q$
\[
\int_Q M\ci Q w_k \dd x \le A \int_Q  w_k \dd x. 
\]
Adding these inequalities we get the conclusion of the lemma. 
\end{proof}

\section{Weighted estimates of vector valued %square functions and sparse 
operators}
\label{s:NewSquareF}

Let $W$, $V$ be matrix weights. 
We want to estimate the norm operator $V^{1/2} T W^{1/2}$ (in the non-weighted $L^2=L^2(\R^N;\R^d)$, where $T$ is either an $\omega$-\cz operator or a big Haar shift or a paraproduct. Since such operators are dominated by convex body sparse operators, it is sufficient to estimate the operators $V^{1/2} T\ci \cS W^{1/2}$, where $T\ci \cS$ is a sparse  integral operator (meaning that $\cS$ is a sparse family of cubes)
\[
T\ci\cS f (x) =\sum_{Q\in\cS} \int_{Q} K\ci Q(x,y) f(y) dy, 
\]
where $K\ci Q$ is supported on $Q\times Q$ and satisfies $\|K\ci Q\|\le |Q|^{-1}$ there, see Section \ref{s:CB-int} for details. We need to estimate operators for all possible choices of kernels $K\ci Q$, and clearly it is sufficient to estimate  the following Lerner type operator $L=L\ci\cS$, 
\begin{align}
\label{SpOp}
Lf(x) & = \sum_{Q\in\cS} \Bigl( |Q|^{-1}\int_Q \nm V^{1/2}(x) W^{1/2}(y) f(y) \nm \dd y \Bigr)\1\ci Q(x)
\\  \notag
  & = \sum_{Q\in\cS} \La \nm  V(x)^{1/2} W^{1/2} f\nm  \Ra\ci  Q\1\ci Q(x)
\end{align}
in the unweighted $L^2$. Since, as we discussed in Section \ref{s:sparse_fam},   a general sparse operator can be dominated by $3^N$ dyadic sparse operators operators, it is sufficient to consider only dyadic sparse operators.

\subsection{Some square functions and sparse operators}
Let $\cS\subset \cD$ be a dyadic sparse sequence. 
Consider the following sparse square functions:

%\begin{enumerate}
%%\item Maximal function $M:L^2(\F^d)\to L^2$, 
%%%%
%%\begin{align}
%%\label{maxFn}
%%Mf (x) &=\sup_{Q\in\cD: x\in Q} \La \nm  V(x)^{1/2} W^{1/2} f\nm  \Ra\ci  Q
%%\\
%%\notag
%%&=  \sup_{Q\in\cD: x\in Q} |Q|^{-1} \int_Q \nm  V^{1/2} (x) W^{1/2} (y) f(y)\nm  dy
%%\end{align}
%%%. 
%%\item Sparse (Lerner type) operator $L$, 
%%%%
%%\begin{align}
%%\label{SpOp}
%%Lf(x) = \sum_{Q\in\cS} \La \nm  V(x)^{1/2} W^{1/2} f\nm  \Ra\ci  Q\1\ci Q(x)
%%\end{align}
%%%%
%\item Sparse square functions 
%%
\begin{align}
\label{SqFn-01}
S_1f(x) &:= \left( \sum_{Q\in\cS}   \La \nm  V(x)^{1/2} W^{1/2} f\nm  \Ra\ci  Q^2\1\ci Q(x)\right)^{1/2} 
\\
\label{SqFn-02}
S_2f(x) &:= \left( \sum_{Q\in\cS}   \La \nm  \La W \Ra\ci Q^{-1/2} W^{1/2} f\nm  \Ra\ci  Q^2\1\ci Q(x)\right)^{1/2}
\\
\label{SqFn-03}
S_3f(x) &:= \left( \sum_{Q\in\cS}   \La \nm  \La V \Ra\ci Q^{1/2} W^{1/2} f\nm  \Ra\ci  Q^2\1\ci Q(x)\right)^{1/2}
\end{align}
%%
%\end{enumerate}
We also have \emph{scalar} versions of the square functions, acting on scalar-valued functions
\begin{align}
\label{SqFn-01-sc}
\wt S_1f(x) &:= \left( \sum_{Q\in\cS}   \La \nm  V(x)^{1/2} W^{1/2} \nm  \cdot |f|\Ra\ci  Q^2\1\ci Q(x)\right)^{1/2} 
\\
\label{SqFn-02-sc}
\wt S_2f(x) &:= \left( \sum_{Q\in\cS}   \La \nm  \La W \Ra\ci Q^{-1/2} W^{1/2} \nm  \cdot |f| \Ra\ci  Q^2\1\ci Q(x)\right)^{1/2}
\\
\label{SqFn-03-sc}
\wt S_3f(x) &:= \left( \sum_{Q\in\cS}   \La \nm  \La V \Ra\ci Q^{1/2} W^{1/2} \nm  \cdot |f| \Ra\ci  Q^2\1\ci Q(x)\right)^{1/2}
\end{align}
and the corresponding scalar version of the Lerner operator
\begin{align}
\label{SpOp-sc}
\wt Lf(x) = \sum_{Q\in\cS} \La \nm  V(x)^{1/2} W^{1/2} \nm  \cdot |f|  \Ra\ci  Q\1\ci Q(x)
\end{align}
%%

%Sparse operator $L$ dominates Fedya's operator $L$, and dominated by our sparse operator \eqref{LS}. 

Also, the vector sparse operators are dominated by their scalar versions, $\nm  S_k f (x)\nm \le \wt S_k \nm  f\nm (x)$, and $\nm  L f (x)\nm \le \wt L \nm  f\nm (x)$

We will need the following well-known lemma. 

\begin{lm}[Carleson Embedding Theorem]
\label{l-dCET}
Let $\mu$ be a Radon measure on $\R^N$ and let $a\ci I\ge 0$, $I\in\cD$ satisfy the Carleson measure condition
\begin{align}
\label{eq-carl}
\sum_{I\in\cD:\, I\subset J} a\ci I \le A \mu(J) . 
\end{align}
Then for any measurable $f\ge 0$ and for any $p\in(1, \infty)$
\begin{align*}
\sum_{I\in\cD }  \left( \mu(I)^{-1} \int_I f\, d\mu \right)^{p} a\ci I \le (p')^p A \|f\|^p\ci{L^p(\mu)}
\end{align*}
\end{lm}

This lemma (with some constant $C(p)$ instead of $(p')^p$) is well-known. The explanation of why it holds with constant $(p')^p$, i.e., with the same constant as in the $L^p$ estimate of the martingale maximal function is explained, for example in \cite[S.~4]{Tr-PosDy}. A direct proof of this lemma via Bellman function approach is also presented in \cite{Lai_CET}. 

\subsection{Weighted estimates of sparse square functions and sparse operators}

In this section we deal only with dyadic operators, and all our $A_2$ and $A_\infty$ conditions are the dyadic ones, like $A_\infty^\cD$, $\bA_2^\cD$ etc. 
We skip the index $\cD$ to simplify the writeup. 

Also to simplify the writing and reading we will use the following notation in skipping variables. 
For a vector valued or matrix-valued function $F$ the symbol $\nm F\nm$ will denote the function $x\mapsto \nm F(x)\nm$. The symbol $\la \nm F\nm \ra\ci Q$ denotes the average of this function. 

For example, 
\[
\La \nm  \La W \Ra\ci Q^{-1/2} W^{1/2} \nm \cdot|f| \Ra\ci  Q = |Q|^{-1} \int_Q \La \nm  \La W \Ra\ci Q^{-1/2} W^{1/2}(y) \nm \cdot|f(y)| \, \dd y, 
\]
and 
\[
 \La \nm  V(x)^{1/2} W^{1/2} \nm  \cdot|f|\Ra\ci  Q  =
 |Q|^{-1} \int_Q \La \nm  \La V(x) \Ra\ci Q^{1/2} W^{1/2}(y) \nm \cdot|f(y)| \, \dd y
\]
%For a scalar weight $w$ let its $A_\infty$ characteristic be the one define in terms of maximal function, 
%%%
%\begin{align*}
%[w]\ci{A_\infty}:= \sup_{Q\in\cD} \La M(\1\ci Q w) \Ra\ci Q /\La w \Ra\ci Q . 
%\end{align*}
%%%
%For a matrix weight $W$ we define its \emph{scalar} $A_\infty$ characteristic as 
%%%
%\begin{align*}
%[W]\ci{A_\infty}\ut{sc} =\sup_{e\in\F^d} [w_e]\ci{A_\infty}, 
%\end{align*}
%%%
%where for $e\in\F^d$ the scalar weight $w_e$ is defined by $w_e(x) = (W(x), e, e)$. 

\begin{lm}
\label{l:S-02}
Let $d\times d $ matrix weight $W$ satisfy the scalar $A_\infty$ condition, and let $\cS$ be a $\lambda$-Carleson dyadic family in the sense of Definition \ref{df:dyCarlFam}. 
Then 
\begin{align*}
\| \wt S_2\| \ci{L^2\to L^2}^2 \le {C}{\lambda} 2^N d [W]\ci{A_\infty}\ut{sc}, 
\end{align*}
where $C$ is an absolute constant, $N$ is the dimension of the underlying space $\R^N$.
%, and $\kappa$ is the sparseness  parameter of $\cS$. .
\end{lm}
\begin{proof}
The proof uses the reverse H\"{o}lder inequality for scalar $A_\infty$ weights. 
%is the same as the proof of Lemma \ref{MF}. 
Let $r>2$, $1/r'+1/r=1$. 
By H\"{o}lder we have 
\begin{align}
\label{13.05}
\La \nm  \La W \Ra\ci Q^{-1/2} W^{1/2} \nm \cdot|f| \Ra\ci  Q  
\le  \la \nm \la W\ra_Q^{-1/2} W^{1/2} \nm ^r \ra_Q^{1/r} \cdot \la | f |^{r'} \ra_Q^{1/r'}\,.
\end{align}

Choose $r= 2(1+\delta)$, where $\delta = 1/(2^{N+1}[W]_{A_\infty}\ut{sc})$, so 
for any $e\in \R^d$ the scalar weight $w_e =(We,e)$ satisfies the reverse H\"{o}lder inequality 
\begin{align}
\label{RH-02}
\La w_e^{1+\delta} \Ra\ci Q \le 2 \La w_e \Ra\ci Q^{1+\delta} , 
\end{align}
see Theorem \ref{RH}. 

We can estimate
\begin{align*}
\nm  \La W \Ra\ci Q^{-1/2} W^{1/2} (y)\nm ^2 \le \nm  W(y)^{1/2} \La W \Ra\ci Q^{-1/2}\nm ^2_{\fS_2}
=\trace ( \La W \Ra\ci Q^{-1/2}    W(y)  \La W \Ra\ci Q^{-1/2})
\end{align*}
The scalar weight $w$, $w(y)= \trace ( \La W \Ra\ci Q^{-1/2}    W(y)  \La W \Ra\ci Q^{-1/2})$ is then an $A_\infty$ weight with $[w]\ci{A_\infty} \le [W]\ci{A_\infty\ut{sc}}$ (as a sum of $d$ such weights). 

Therefore, replacing in \eqref{13.05} the operator norm by the Hilbert--Schmidt norm we can apply the reverse H\"{o}lder  to get 
\begin{align*}
 \La \nm  \La W \Ra\ci Q^{-1/2} W^{1/2} \nm \cdot|f| \Ra\ci  Q 
\le  
2 \la \nm \la W\ra_Q^{-1/2} W^{1/2} \nm _{\fS_2}^2 \ra_Q^{1/2} \cdot \la \nm  f\nm ^{r'} \ra_Q^{1/r'}\,.
\end{align*}
But
\begin{align}
\label{trace-01}
\La\nm \la W\ra_Q^{-1/2} W^{1/2} \nm _{\fS_2}^2\Ra =
\trace |Q|^{-1} \int_Q \la W\ra_Q^{-1/2} W(y) \la W\ra_Q^{-1/2} \dd y =
\trace I_d =d
\end{align}
so %%
\begin{align*}
\La \nm  \La W \Ra\ci Q^{-1/2} W^{1/2} \nm \cdot|f| \Ra\ci  Q  
\le  
2 d^{1/2} \la \nm  f\nm ^{r'} \ra_Q^{1/r'}\,.
\end{align*}
Therefore
\begin{align*}
\|   \La \nm  \La W \Ra\ci Q^{-1/2} W^{1/2} \nm \cdot|f| \Ra\ci  Q \1\ci Q\|\ci{L^2}^2 \le 4d \la|f|\ra\ci Q^{2/r'} |Q|, 
\end{align*}
so
\begin{align*}
\| \wt S_2 f\|\ci{L^2}^2 & = \sum_{Q\in\cS} \|   \La \nm  \La W \Ra\ci Q^{-1/2} W^{1/2} \nm \cdot|f| \Ra\ci  Q \1\ci Q\|\ci{L^2}^2
\\
& \le 4d \sum_{Q\in\cS} \la|f|\ra\ci Q^{2/r'} |Q| .
\end{align*}
%%
%So, we need to estimate 
%$
%\sum_{Q\in\cS} \la \nm  f\nm ^{r'} \ra_Q^{2/r'} |Q|
% \le \frac1\eta \sum_{Q\in\cS} \la \nm  f\nm ^{r'} \ra_Q^{2/r'} |E\ci Q|
%$. 

To estimate the last sum denote $p=2/r'$, $\f(x) = \nm  f(x)\nm ^{r'}$, and apply Lemma \ref{l-dCET} with $\mu$ being the Lebesgue measure and 
\[
a\ci Q =\left\{ \begin{array}{ll} | Q| \qquad & Q\in\cS ;\\ 0 & Q\notin \cS. \end{array} \right.
\] 
We get  
\begin{align*} 
\sum_{Q\in\cS} \la \nm  f\nm ^{r'} \ra_Q^{2/r'} | Q|  =  \sum_{Q\in\cS} \la \f \ra_Q^{p} |Q|
 \le \lambda (p')^p \| \f\| \ci{L^p}^p =
\lambda (p')^p \|  f\| \ci{L^2}^2 \,.
\end{align*}
%%
%where $\kappa$ is the sparseness characteristic of $\cS$. 
Direct computations show 
\begin{align*}
p' = 2+1/\delta \le 3/\delta, \qquad p=\frac{1+2\delta}{1+\delta} \le 1+\delta
\end{align*}
so 
\begin{align*}
(p')^p \le (3/\delta)^{1+\delta} \le C /\delta =C 2^{N+1} [W]\ci{A_\infty\ut{sc}}
\end{align*}
where $C$ is an absolute constant (maximum of the function $\delta\mapsto 3(3/\delta)^{\delta}$ on $(0, 1]$). 
\end{proof}

\begin{lm}
\label{l:S-03}
Let $d\times d $ matrix weight $W$ satisfy the scalar $A_\infty$ condition, and let the weights $V$ and $W$ satisfy the two weight matrix $\mathbf A_2$ condition. Let also $\cS$ be a $\lambda$-Carleson dyadic family in the sense of Definition \ref{df:dyCarlFam}. 
Then 
\begin{align*}
\| \wt S_3\| \ci{L^2\to L^2}^2 \le {C}{\lambda} 2^N d \cdot[W,V]\ci{\mathbf A_2}[W]\ci{A_\infty}\ut{sc}, 
\end{align*}
where $C$ is an absolute constant, $N$ is the dimension of the underlying space $\R^N$.
%, and $\kappa$ is the sparseness  parameter of $\cS$. .
\end{lm}

\begin{proof}
To prove this lemma one can just rewrite the proof of Lemma \ref{l:S-02} word by word, replacing each occurrence of $\La W \Ra\ci Q^{-1/2}$ by $\La V\Ra\ci Q^{1/2}$. The only difference will be that instead of \eqref{trace-01} we will have 
\begin{align}
\label{trace-02}
\La\nm \la V\ra_Q^{1/2} W^{1/2} \nm _{\fS_2}^2\Ra\ci Q  & =
\trace |Q|^{-1} \int_Q \la V\ra_Q^{1/2} W(y) \la V\ra_Q^{1/2} \dd y
\\ 
\notag
& =
\trace (\la V\ra_Q^{1/2} \la W\ra_Q \la V\ra_Q^{1/2}  \le d \cdot[W,V]\ci{\mathbf A_2}, 
\end{align}
which accounts for an extra factor $[W,V]\ci{\mathbf A_2}$. 
\end{proof}

\begin{lm}
\label{l:S-01}
Let the $d\times d $ matrix weight $W$ satisfy the scalar $A_\infty$ condition, and let the weights $V$ and $W$ satisfy the two weight matrix $\mathbf A_2$ condition. Let also $\cS$ be a $\lambda$-Carleson dyadic family in the sense of Definition \ref{df:dyCarlFam}. 

Then for the square function $\wt S_1$ defined by  \eqref{SqFn-01}
\begin{align*}
\| \wt S_1\| \ci{L^2\to L^2}^2 \le {C}\lambda 2^N d\cdot [W,V]\ci{\mathbf A_2} [W]\ci{A_\infty}\ut{sc}, 
\end{align*}
where $C$ is an absolute constant, $N$ is the dimension of the underlying space $\R^N$.
\end{lm}

\begin{proof}
The proof is similar to the proof of Lemma \ref{l:S-02}.  Instead of \eqref{13.05} we write
\begin{align}
\label{13.07}
 \La \nm  V(x)^{1/2} W^{1/2} \nm  \cdot|f|\Ra\ci  Q 
\le  \la \nm  V(x)^{1/2} W^{1/2} \nm ^r \ra_Q^{1/r}  \la | f |^{r'} \ra_Q^{1/r'}\,,
\end{align}
where $r$ is the same as in the proof of Lemma \ref{l:S-02}. 

Then we notice that for any fixed $x$ the weight $w$, 
\[
w(y)= \nm  W(y)^{1/2} V(x)^{1/2}\nm _{\fS_2}^2 =\trace (V(x)^{1/2} W(y) V(x)^{1/2}) 
\]
satisfies the $A_\infty$ condition with $[w]\ci{A_\infty}\le [W]\ci{A_\infty\ut{sc}}$. 
Therefore we can use the reverse H\"{o}lder inequality \eqref{RH-02} to get from \eqref{13.07}
\begin{align}
\label{13.08}
 \La \nm  V(x)^{1/2} W^{1/2} \nm  \cdot|f|\,\Ra\ci  Q 
\le  2 \la \nm  V(x)^{1/2} W^{1/2} \nm ^2_{\fS_2} \ra_Q^{1/2} \cdot \la | f |^{r'} \ra_Q^{1/r'}\,.
\end{align}
But 
\begin{align*}
\nm  V(x)^{1/2} W(y)^{1/2} \nm _{\fS_2}^2 = \trace (V(x)W(y))
\end{align*}
so
\begin{align}
\label{trace-03}
\int_Q \La\nm  V(x)^{1/2} W^{1/2} \nm ^2_{\fS_2} \Ra\ci{Q} \dd x & =
|Q|^{-1} \iint_{Q\times Q} \trace(V(x)W(y)) \dd x \dd y 
\\ \notag  &= 
|Q| \trace(\la W\ra\ci Q \la V\ra\ci Q ) 
\\
\notag
&= |Q| \nm  \la W\ra\ci Q^{1/2}\la V\ra\ci Q^{1/2}\nm _{\fS_2}^2 \le |Q|d\cdot [V, W]\ci{\mathbf A_2}.  
\end{align}

Using the above inequality we can estimate
\begin{align*}
\| \wt S_1 f\|\ci{L^2}^2 & = \sum_{Q\in\cS}  \int_Q \la \nm V^{1/2}(x) W^{1/2} \nm \cdot |f|\ra\ci Q^2 \dd x  &&
\\
& \le 4  \sum_{Q\in\cS} \la |f|^{r'}\ra^{2/r'}\ci Q \int_Q \la \nm  V(x)^{1/2} W^{1/2} \nm ^2_{\fS_2} \ra\ci Q  \dd x  && \text{by \eqref{13.08}}\\
& \le 4 d\cdot [V, W]\ci{\mathbf A_2} \sum_{Q\in\cS} \la |f|^{r'}\ra^{2/r'}\ci Q |Q| && \text{by \eqref{trace-03}}
\end{align*}
So, it remains to estimate $\sum_{Q\in\cS} \la\nm  f\nm ^{r'}\ra^{2/r'}|Q|$. But it  is already done in the proof of Lemma \ref{l:S-02}, where it is shown that 
\[
\sum_{Q\in\cS} \la\nm  f\nm ^{r'}\ra^{2/r'}|Q| \le C \lambda 2^{N+1} [W]\ci{A_\infty\ut{sc}}, 
\] 
so we are done. 
\end{proof}

%\begin{rem}
%\label{r:S-03}
%In the above Lemma \ref{l:S-01} the square function $\wt S_1$ can be replaced by the square function $\wt S_3$ from \eqref{SqFn-03-sc}. The proof will remain essentially the same with obvious modifications. 
%\end{rem}

%\begin{rem}
%\label{r:scalar}
%In fact, all the estimates above (Lemmas \ref{l:S-02}, \ref{l:S-01}) hold for modified square functions $\wt S_k$, acting on scalar functions. The square function $\wt S_1$ is defined 
%%%
%\begin{align}
%\label{SqFn-01-sc}
%S_1f(x) &:= \left( \sum_{Q\in\cS}   \La \nm  V(x)^{1/2} W^{1/2}  \nm  \cdot |f| \Ra\ci  Q^2\1\ci Q(x)\right)^{1/2} , 
%\end{align}
%and other functions are defined similarly. 
%%%
%\end{rem}
\begin{rem}
In the definition of the square functions we can replace summation over a sparse sequence by the summation with Carleson weights. For example, instead of $\wt S_1$ in \eqref{SqFn-01-sc} we can consider
\begin{align}
\label{SqFn-01-c}
S_1f(x) &:= \left( \sum_{Q\in\cD} a\ci Q  \La \nm  V(x)^{1/2} W^{1/2}\nm \cdot | f |  \Ra\ci  Q^2\1\ci Q(x)\right)^{1/2} 
\end{align}
where $a =\{a\ci Q\}\ci{Q\in\cD}$, $a\ci Q\ge 0$ is a $\lambda$-Carleson sequence, 
\begin{align}
\label{carl-01}
\sup_{R\in\cD} |R|^{-1} \sum_{Q\in\cD(R)} a\ci Q |Q| \le \lambda <\infty;  
\end{align}
similarly for all other square functions. 

Lemmas \ref{l:S-02}, \ref{l:S-03} and  \ref{l:S-01} with absolutely the same proofs will hold for these square functions. 
\end{rem}

\begin{lm}
\label{l:L-sc}
Let $d\times d$ matrix weights $V$ and $W$ satisfy the scalar $A_\infty\ut{sc}$ condition, and let them satisfy the joint $\mathbf A_2$ condition. Then the norm of the scalar Lerner operator $\wt L$  satisfies the estimate 
\begin{align*}
 \| \wt L\| \ci{L^2\to L^2}\le \frac{C}{1-\kappa} 2^N d \cdot [W,V]\ci{\mathbf A_2}^{1/2} [W]\ci{A_\infty\ut{sc}}^{1/2}[V]\ci{A_\infty\ut{sc}}^{1/2}
\end{align*}
\end{lm}

\begin{proof}
Take $f, g\in L^2$ and let us estimate $(\wt L f, g)\ci{L^2}$. Without loss of generality we can assume that $f,g\ge 0$. So
\begin{align*}
(\wt L f, g)\ci{L^2} &=\sum_{Q\in\cS}|Q|^{-1} \iint_{Q\times Q} f(y) g(x) \nm  V(x)^{1/2} W(y)^{1/2} \nm  \dd x \dd y 
\\
&\le 
\sum_{Q\in\cS}|Q|^{-1} \iint_{Q\times Q}  g(x) \nm  V(x)^{1/2}\la V\ra\ci Q^{-1/2}\nm \cdot  \nm  \la V\ra\ci Q^{1/2} W(y)^{1/2} \nm  f(y) \dd x \dd y 
\\ 
&= \sum_{Q\in\cS} \la \nm  V^{1/2} \la V\ra\ci Q^{-1/2}\nm  \cdot f \ra\ci Q 
\la \nm   \la V\ra\ci Q^{1/2}W\nm  \cdot g \ra\ci Q |Q|
\\
& \le \| \wt S_3 g \| \ci{L^2} \| \wt S_2^V f\| \ci{L^2}\,,
\end{align*}
where $\wt S_2^V$ is the  scalar  square function \eqref{SqFn-02-sc} with $W$ replaced by $V$. Combining estimates for the norms of $\wt S_2$ and $\wt S_3$ from Lemma \ref{l:S-02} and Lemma \ref{l:S-03} respectively we get the conclusion of the lemma. 
\end{proof}

%\begin{rem*}
%The above lemma gives the same estimate for the norm of the operator $L$ in \eqref{LS}. 
%\end{rem*}

\subsection{A better estimate for a simple sparse family}

If the sparse sequence $\cS$ has a very simple structure, we can get a better estimate for the norm of $\wt L$. 

\begin{df*}
A sparse family $\cS\subset\cD$ is called \emph{simple} if each cube $Q\in\cS$ has at most one $\cS$-child.

Note that in a simple sparse family $\cS$ all cubes $Q\in\cS$ except the minimal (by inclusion) one  have exactly one $\cS$-child; the minimal cube (if such one exists) has no $\cS$-children. 
\end{df*}

\begin{lm}
\label{l:L-sc-simple}
Let $\cS\subset \cD$ be a simple sparse family of cubes. Assume also that
$d\times d$ matrix weights $V$ and $V$ satisfy the scalar $A_\infty\ut{sc}$ condition, and that they satisfy the joint $\mathbf A_2$ condition.

Then for the corresponding sparse (Lerner) operator $\wt L$ defined by \eqref{SpOp-sc} 
\begin{align*}
\| \wt L \| \ci{L^2\to L^2} \le C 2^{N/2} d^{1/2} [W,V]\ci{\mathbf A_2}^{1/2} \left(([W]\ci{A_\infty}\ut{sc})^{1/2} + ([V]\ci{A_\infty}\ut{sc})^{1/2}\right) .
\end{align*}
%%
%(i.e.~it has the same estimate as the square function $\wt S_2$ in Lemma \ref{l:S-01} with $\kappa=1/2$)
\end{lm}

\begin{proof}
Note first that for a simple sparse family its sparseness characteristic $\e$ satisfies $\e \le 2^{-N} \le 1/2$. 

Let us estimate $(\wt L f, g)$, $f, g\in L^2$, $\| f\| \ci{L^2} = \| g\| \ci{L^2} =1$. Without loss of generality we can assume that $f, g\ge 0$. We have 
\begin{align}
\label{Tfg-01}
(\wt L f, g) = \sum_{Q\in\cS} |Q|^{-1} \iint_{Q\times Q} f(y)\nm  W(y)^{1/2} V(x)^{1/2} \nm  g(x) dx dy. 
\end{align}

It is sufficient to prove this lemma for \emph{finite} simple families, 
so let us assume that our simple sparse family $\cS$ is finite. Then the operator $\wt L$ is bounded (finite sum of bounded terms), so given $\e>0$ we can pick $f, g\in L^2$, $f, g\ge 0$, $\| f\| \ci{L^2}=  \| g\| \ci{L^2} =1$, so that $(\wt L f, g) \ge (1-\e)\| \wt L\| \ci{L^2\to L^2}$.

For $Q\in\cS$ let $\wh Q$ be the $\cS$-child of $Q$, and let $E\ci Q:= Q\setminus \wh Q$. Then for each $Q$ the integral over $Q\times Q$ in \eqref{Tfg-01} can be split into 3 integrals, 
\begin{align*}
\iint_{Q\times Q}\ldots = \iint_{Q\times E\ci Q}\ldots + \iint_{E\ci Q\times \wh Q} \ldots+ \iint_{\wh Q\times \wh Q}\,.
\end{align*}

The sum of the first integrals can be estimated by the square function $(\wt S_1 f, g)$, so by Lemma \ref{l:S-01} it can be estimated by $C 2^{N/2} d^{1/2} [W,V]\ci{\mathbf A_2}^{1/2} ([W]\ci{A_\infty}\ut{sc})^{1/2}$. 

The sum of the second integrals is dominated by the sum of integrals over $E\ci Q\times Q$, so Lemma \ref{l:S-01} with $V$ and $W$ interchanged gives the estimate $C 2^{N/2} d^{1/2} [W,V]\ci{\mathbf A_2}^{1/2} ([V]\ci{A_\infty}\ut{sc})^{1/2}$. 

Now let us consider the last sum:
\begin{align*}
\sum_{Q\in\cS} |Q|^{-1}& \iint_{\wh Q\times \wh Q} f(y)\nm  W(y)^{1/2} V(x)^{1/2} \nm  g(x) \dd x \dd y 
\\ &
=2^{-N} \sum_{Q\in\cS} |\wh Q|^{-1} \iint_{\wh Q\times \wh Q} f(y)\nm  W(y)^{1/2} V(x)^{1/2} \nm  g(x) \dd x \dd y 
\\
&\le \frac12 \sum_{Q\in\cS} | Q|^{-1} \iint_{Q\times Q} f(y)\nm  W(y)^{1/2} V(x)^{1/2} \nm  g(x) \dd x \dd y
\\ &
\le \frac12 \| \wt L\|\ci{L^2\to L^2}.
\end{align*}
The standard ``pulling out by hair'' argument completes the proof. 
\end{proof}

\section{Some remarks}
There are several speculations here. 
\begin{enumerate}
\item Estimates are very rough, they do not use any intricacies of the matrix case. We do not use any matrix Carleson embedding theorems here. But we still cannot get a better estimate even for a simple sparse operator (like $V^{1/2}L W^{1/2}$, where $L$ is the usual scalar sparse operator). 

\item Examples showing that linear in the $A_2$ characteristic of the scalar weight is optimal can be obtained by considering weights with one singularity (say behaving like $|x|^p$, $0<p<1$) and estimating the norms of the Hilbert transform as $p\to 1^-$. 

Note, that the same example gives the optimal lower bound for a simple sparse operator with the sparse family $[0, 2^{-n})$, $n\in\N$. So, if we want to get a counterexample to the matrix linear $A_2$ conjecture, we need something more complicated than weights with simple singularities and simple sparse operators. 

%\item Finally, we do not know if one can get linear $A_2$ estimate of the maximal function \eqref{maxFn}. It is possible to get the estimate with exponent $3/2$, like the estimate in Lemma \ref{l:L-sc}. 
\end{enumerate}

\end{document}